\documentclass{amsart}
\usepackage{hyperref}
\newtheorem{definition}{Definition}[section]
\newtheorem{lemma}[definition]{Lemma}
\newtheorem{theorem}[definition]{Theorem}
\newtheorem{proposition}[definition]{Proposition}

\newcommand{\C}{\mathbb{C}}
\newcommand{\D}{\mathbb{D}}
\newcommand{\N}{\mathbb{N}}
\newcommand{\R}{\mathbb{R}}
\newcommand{\T}{\mathbb{T}}
\newcommand{\Z}{\mathbb{Z}}

\newcommand{\cA}{\mathcal{A}}
\newcommand{\cP}{\mathcal{P}}
\newcommand{\cR}{\mathcal{R}}

\newcommand{\im}{\mathrm{Im}\,}
\newcommand{\tr}{\mathrm{tr}\,}
\newcommand{\eps}{\varepsilon}

\begin{document}
\title[Asymptotic Formulas for Traces of Toeplitz Matrices]
{Higher Order Asymptotic Formulas\\
for Traces of Toeplitz Matrices with \\
Symbols in H\"older-Zygmund Spaces}

\author{Alexei~Yu.~Karlovich}
\thanks{This work is supported by Centro de Matem\'atica da Universidade do Minho
(Portugal) and by the Portuguese Foundation of Science and Technology through the
research program POCTI}
%F.C.T. (Portugal) grant FCT/FEDER/POCTI/MAT/ 59972/2004}
\address{%
Universade do Minho,
Centro de Matem\'atica,
Escola de Ci\^encias,
Campus de Gualtar,
4710-057, Braga,
Portugal}
\email{oleksiy@math.uminho.pt}

\address{%
\textit{Current address:}
Departamento de Matem\'atica,
Instituto Superior T\'ecnico,
Av. Rovisco Pais 1,
1049--001, Lisbon, Portugal}
\email{akarlov@math.ist.utl.pt}

\begin{abstract}
We prove a higher order asymptotic formula for traces of finite block Toeplitz
matrices with symbols belonging to H\"older-Zygmund spaces. The remainder
in this formula goes to zero very rapidly for very smooth symbols. This
formula refine previous asymptotic trace formulas by Szeg\H{o} and Widom and
complement higher order asymptotic formulas for determinants of finite block
Toeplitz matrices due to B\"ottcher and Silbermann.
\end{abstract}
\maketitle
%%%%%%%%%%%%%%%%%%%%%%%%%%%%%%%%%%%%%%%%%%%%%%%%%%%%%%%%%%%%%%%%%%%%%%%%%%
\section{Introduction and main result}
\subsection{Finite block Toeplitz matrices}
Let $\Z,\N,\Z_+$, and $\C$ be the sets of integers, positive integers,
nonnegative integers, and all complex numbers, respectively.
Suppose $N\in\N$. For a Banach space $X$,
let $X_N$ and $X_{N\times N}$
be the spaces of vectors and matrices with entries in $X$.
Let $\T$ be the unit circle. For $1\le p\le\infty$, let
$L^p:=L^p(\T)$ and $H^p:=H^p(\T)$ be the standard Lebesgue and Hardy
spaces of the unit circle. For $a\in L_{N\times N}^1$ one can define
\[
a_k=\frac{1}{2\pi}\int_0^{2\pi}a(e^{i\theta})e^{-ik\theta}d\theta
\quad (k\in\Z),
\]
the sequence of the Fourier coefficients of $a$.
Let $I$ be the identity operator, $P$ be the Riesz projection of
$L^2$ onto $H^2$, $Q:=I-P$, and define $I,P$, and $Q$ on
$L_N^2$ elementwise. For $a\in L_{N\times N}^\infty$ and $t\in\T$,
put $\widetilde{a}(t):=a(1/t)$ and $(Ja)(t):=t^{-1}\widetilde{a}(t)$.
Define \textit{Toeplitz operators}
\[
T(a):=PaP|\im P,
\quad
T(\widetilde{a}):=JQaQJ|\im P
\]
and \textit{Hankel operators}
\[
H(a):=PaQJ|\im P,
\quad
H(\widetilde{a}):=JQaP|\im P.
\]
The function $a$ is called the \textit{symbol} of $T(a)$, $T(\widetilde{a})$,
$H(a)$, $H(\widetilde{a})$. We are interested in the asymptotic behavior of
\textit{finite block Toeplitz matrices}
$T_n(a)=[a_{j-k}]_{j,k=0}^n$ generated by (the Fourier coefficients of)
the symbol $a$ as $n\to\infty$. Many results in this direction are contained
in the books by Grenander and Szeg\H{o} \cite{GS58}, B\"ottcher and Silbermann
\cite{BS83,BS99,BS06}, Simon \cite{Simon05}, and B\"ottcher and Grudsky
\cite{BG05}.
%%%%%%%%%%%%%%%%%%%%%%%%%%%%%%%%%%%%%%%%%%%%%%%%%%%%%%%%%%%%%%%%%%%%%%%%%%%%
\subsection{Szeg\H{o}-Widom limit theorems}
Let us formulate precisely the most relevant results.
Let $K^2_{N\times N}$ be the Krein algebra \cite{Krein66}
of matrix functions $a$ in $L_{N\times N}^\infty$
satisfying
\[
\sum_{k=-\infty}^\infty \|a_k\|^2(|k|+1)<\infty,
\]
where $\|\cdot\|$ is any matrix norm on $\C_{N\times N}$.
The following beautiful theorem about the asymptotics of finite
block Toeplitz matrices was proved by Widom \cite{Widom76}.
%%%
\begin{theorem}\label{th:Widom1}
{\rm (see \cite[Theorem~6.1]{Widom76})}.
If $a\in K_{N\times N}^2$ and the Toeplitz operators
$T(a)$ and $T(\widetilde{a})$ are invertible on $H_N^2$, then
$T(a)T(a^{-1})-I$ is of trace class and, with appropriate branches
of the logarithm,
%%%
\begin{equation}\label{eq:Widom1}
\log\det T_n(a)=(n+1)\log G(a)+\log \det T(a)T(a^{-1})+o(1)
\quad\mbox{as}\quad n\to\infty,
\end{equation}
%%%
where
\begin{equation}\label{eq:defG}
G(a):=\lim_{r\to 1-0}\exp\left(\frac{1}{2\pi}\int_0^{2\pi}
\log\det\widehat{a}_r(e^{i\theta})d\theta\right),
\widehat{a}_r(e^{i\theta}):=\sum_{n=-\infty}^{\infty}a_nr^{|n|}e^{in\theta}.
\end{equation}
%%%
\end{theorem}
%%%%%%%%%%%%%%%%%%%%%%%%%%%%%%%%%%%%%%%%%%%%%%%%%%%%%%%%%%%%%%%%%%%%%%%%%%%%%%
In formula (\ref{eq:Widom1}),  $\det T(a)T(a^{-1})$ refers to the determinant
defined for operators on Hilbert space differing from the identity by an
operator of trace class \cite[Chap.~4]{GK69}.

The proof of the above result in a more general form is contained in
\cite[Theorem~6.11]{BS83} and \cite[Theorem~10.30]{BS06}
(in this connection see also \cite{Ehrhardt03}).

Let $\lambda_1^{(n)},\dots,\lambda_{(n+1)N}^{(n)}$ denote the eigenvalues of
$T_n(a)$ repeated according to their algebraic multiplicity. Let $\mathrm{sp}\,A$
denote the spectrum of a bounded linear operator $A$ and $\tr M$ denote the trace
of a matrix $M$. Theorem~\ref{th:Widom1} is equivalent to the assertion
\[
\sum_i\log\lambda_i^{(n)}=\tr\log T_n(a)=(n+1)\log G(a)+\log\det T(a)T(a^{-1})+o(1).
\]
Widom \cite{Widom76} noticed that Theorem~\ref{th:Widom1} yields even a description of
the asymptotic behavior of $\tr f(T_n(a))$ if one replaces $f(\lambda)=\log\lambda$
by an arbitrary function $f$ analytic in an open neighborhood of the union
$\mathrm{sp}\,T(a)\cup\mathrm{sp}\,T(\widetilde{a})$ (we henceforth call such
$f$ simply analytic on $\mathrm{sp}\,T(a)\cup\mathrm{sp}\,T(\widetilde{a})$).
%%%%%%%%%%%%%%%%%%%%%%%%%%%%%%%%%%%%%%%%%%%%%%%%%%%%%%%%%%%%%%%%%%%%%%%%%%%%
\begin{theorem}\label{th:Widom2}
{\rm (see \cite[Theorem~6.2]{Widom76}).}
If $a\in K^2_{N\times N}$ and if $f$ is analytic on
$\mathrm{sp}\,T(a)\cup\mathrm{sp}\,T(\widetilde{a})$, then
%%%
\begin{equation}\label{eq:Widom2}
\tr f(T_n(a))=(n+1)G_f(a)+E_f(a)+o(1)
\quad\mbox{as}\quad n\to\infty,
\end{equation}
where
%%%
\begin{eqnarray*}
G_f(a)
&:=&
\frac{1}{2\pi}\int_0^{2\pi}(\tr f(a))(e^{i\theta})d\theta,
\\
E_f(a)
&:=&
\frac{1}{2\pi i}\int_{\partial\Omega}f(\lambda)
\frac{d}{d\lambda}\log\det T[a-\lambda]T[(a-\lambda)^{-1}]d\lambda,
\end{eqnarray*}
%%%
and $\Omega$ is any bounded open set containing
$\mathrm{sp}\,T(a)\cup\mathrm{sp}\,T(\widetilde{a})$
on the closure of which $f$ is analytic.
\end{theorem}
%%%%%%%%%%%%%%%%%%%%%%%%%%%%%%%%%%%%%%%%%%%%%%%%%%%%%%%%%%%%%%%%%%%%%%%%%%%%%
The proof of Theorem~\ref{th:Widom2} for continuous symbols $a$ is also given
in \cite[Section 10.90]{BS06}.
In the scalar case ($N=1$) Theorems~\ref{th:Widom1} and \ref{th:Widom2}
go back to Gabor Szeg\H{o} (see \cite{GS58} and historical remarks
in \cite{BS83,BS99,BS06,Simon05}).
%%%%%%%%%%%%%%%%%%%%%%%%%%%%%%%%%%%%%%%%%%%%%%%%%%%%%%%%%%%%%%%%%%%%%%%%%%%%%
\subsection{H\"older-Zygmund spaces}
Suppose $g$ is a bounded function on $\T$. The \textit{modulus of continuity}
of $g$ is defined for $s\ge 0$ by
\[
\omega_1(g,s):=\sup\big\{|g(e^{i(x+h)})-g(e^{ix})|:\ x,h\in\R,\ |h|\le s\big\}.
\]
By the \textit{modulus of smoothness} (of order $2$) of $g$
is meant the function (see, e.g., \cite[Section~3.3]{Timan63})
defined for $s\ge 0$ by
\[
\omega_2(g,s):=\sup\big\{|g(e^{i(x+h)})-2g(e^{ix})+g(e^{i(x-h)})|:\
x,h\in\R,\ |h|\le s\big\}.
\]
Let $C=C(\T)$ be the set of all continuous functions on $\T$.
Given $\gamma>0$, write $\gamma=m+\delta$, where $m\in\Z_+$ and $\delta\in(0,1]$.
The H\"older-Zygmund space $C^\gamma=C^\gamma(\T)$ is defined
(see, e.g., \cite[Section~3.5.4]{ST87}) by
\[
C^\gamma:=\big\{f\in C: f^{(j)}\in C,\ 1\le j\le m, \ [f^{(m)}]_\delta<\infty\big\}
\]
with the norm
\[
\|f\|_\gamma:=\sum_{j=0}^m\|f^{(j)}\|_\infty+[f^{(m)}]_\delta,
\]
where $f^{(j)}$ is the derivative of order $j$ of $f$, $\|\cdot\|_\infty$
is the norm in $L^\infty$, and
\[
[g]_\delta:=\sup_{s>0}\frac{\omega_2(g,s)}{s^\delta},
\quad 0<\delta\le 1.
\]
Notice that if $\gamma>0$ is not integer, then $[g]_\delta$ can be replaced by
\[
[g]_\delta^*:=\sup_{s>0}\frac{\omega_1(g,s)}{s^\delta},
\quad
0<\delta<1
\]
in the above definition.
%%%%%%%%%%%%%%%%%%%%%%%%%%%%%%%%%%%%%%%%%%%%%%%%%%%%%%%%%%%%%%%%%%%%%%%%%%%%%
\subsection{B\"ottcher-Silbermann higher order asymptotic formulas for determinants}
Following \cite{Widom76} and \cite[Sections~7.5--7.6]{BS06},
for $n\in\Z_+$ and $a\in L_{N\times N}^\infty$ define the operators
$P_n$ and $Q_n$ on $H_N^2$ by
\[
P_n:\sum_{k=0}^\infty a_k t^k
\mapsto
\sum_{k=0}^n a_k t^k,
\quad
Q_n:=I-P_n.
\]
The operator $P_nT(a)P_n:P_nH_N^2\to P_nH_N^2$ may be identified with the finite
block Toeplitz matrix $T_n(a):=[a_{j-k}]_{j,k=0}^n$. For a unital Banach algebra
$A$ we will denote by $GA$ the group of all invertible elements of $A$.
For $1\le p\le \infty$, put
%%%
\begin{eqnarray*}
H_\pm^p
&:=&
\big\{a\in L^p : a_{\mp n}=0
\mbox{ for } n\in\N\big\}.
\end{eqnarray*}

B\"ottcher and Silbermann \cite{BS80} proved among other things the following
result.
%%%%%%%%%%%%%%%%%%%%%%%%%%%%%%%%%%%%%%%%%%%%%%%%%%%%%%%%%%%%%%%%%%%%%%%%%%%%%%
\begin{theorem}\label{th:BS}
Let $p\in\N$ and $\alpha,\beta>0$ satisfy $\alpha+\beta>1/p$. Suppose
$a=u_-u_+$, where $u_+\in G(C^\alpha\cap H_+^\infty)_{N\times N}$ and
$u_-\in G(C^\beta\cap H_-^\infty)_{N\times N}$, and the Toeplitz operator
$T(\widetilde{a})$ is invertible on $H_N^2$. Then

\begin{enumerate}
\item[{\rm (a)}]
there exist $v_-\in G(H_-^\infty)_{N\times N}$ and $v_+\in G(H_+^\infty)_{N\times N}$
such that $a=v_+v_-$;

\item[{\rm (b)}]
there exist a constant $\widetilde{E}(a)\ne 0$ such that
%%%
\begin{eqnarray*}
\log \det T_n(a)
&=&
(n+1)\log G(a)+\log \widetilde{E}(a)
\\
&&
+
\tr\left[
\sum_{\ell=1}^n\sum_{j=1}^{p-1}\frac{1}{j}
\left(\sum_{k=0}^{p-j-1}G_{\ell,k}(b,c)\right)^j
\right]
\\
&&+O(1/n^{(\alpha+\beta)p-1})
\end{eqnarray*}
%%%
as $n\to\infty$, where the correcting terms $G_{\ell,k}(b,c)$ are given by
%%%
\begin{equation}\label{eq:defGnk}
G_{\ell,k}(b,c):=P_0 T(c) Q_\ell
\big(Q_\ell H(b)H(\widetilde{c})Q_\ell\big)^k
Q_\ell T(b)P_0
\quad (\ell,k\in\Z_+)
\end{equation}
%%%
and the functions $b,c$ are given by $b:=v_-u_+^{-1}$ and $c:=u_-^{-1}v_+$.
\end{enumerate}
If, in addition, $p=1$, then
\begin{enumerate}
\item[{\rm (c)}]
the operator $T(a)T(a^{-1})-I$ is of trace class and
%%%
\begin{equation}\label{eq:BS}
\log\det T_n(a)=(n+1)\log G(a)+\log \det T(a)T(a^{-1})+O(1/n^{\alpha+\beta-1})
\end{equation}
%%%
as $n\to\infty$.
\end{enumerate}
\end{theorem}
%%%%%%%%%%%%%%%%%%%%%%%%%%%%%%%%%%%%%%%%%%%%%%%%%%%%%%%%%%%%%%%%%%%%%%%%%%%%%%%%

The sketch of the proof of parts (a) and (b) is contained in
\cite[Sections~6.18(ii)]{BS83} and in \cite[Theorem~10.35(ii)]{BS06}.
Part (c) is explicitly stated in \cite[Section~6.18(ii)]{BS83}
or immediately follows from \cite[Theorems~10.35(ii) and 10.37(ii)]{BS06}.
%%%%%%%%%%%%%%%%%%%%%%%%%%%%%%%%%%%%%%%%%%%%%%%%%%%%%%%%%%%%%%%%%%%%%%%%%%%%%%%
\subsection{Our main result}
Our main result is the following refinement of Theorem~\ref{th:Widom2}.
%%%%%%%%%%%%%%%%%%%%%%%%%%%%%%%%%%%%%%%%%%%%%%%%%%%%%%%%%%%%%%%%%%%%%%%%%%%%%%
\begin{theorem}\label{th:main}
Let $\gamma>1/2$.
If $a\in C^\gamma_{N\times N}$ and  if $f$ is analytic on
$\mathrm{sp}\,T(a)\cup\mathrm{sp}\,T(\widetilde{a})$,
then {\rm(\ref{eq:Widom2})} is true with $o(1)$ replaced by $O(1/n^{2\gamma-1})$.
\end{theorem}
%%%%%%%%%%%%%%%%%%%%%%%%%%%%%%%%%%%%%%%%%%%%%%%%%%%%%%%%%%%%%%%%%%%%%%%%%%%
Clearly, this result is predicted by Theorem~\ref{th:BS}(c) with $\gamma=\alpha=\beta$.
The key point in the Widom's proof of Theorem~\ref{th:Widom2} is that
(\ref{eq:Widom1}) is valid for $a-\lambda$ in place of $a$,
%where $\lambda$ is a parameter in $\mathrm{sp}\,T(a)\cup \mathrm{sp}\,T(\widetilde{a})$,
uniformly with respect to $\lambda$ in a neighborhood of $\partial\Omega$.
We will show that the same remains true for the higher order asymptotic formula
(\ref{eq:BS}) with $\gamma=\alpha=\beta$. In Section~\ref{sec:2}
we collect necessary information about right and left Wiener-Hopf factorizations
in decomposing algebras and mention that a nonsingular matrix function belonging
to a H\"older-Zygmund space $C_{N\times N}^\gamma$ ($\gamma>0$) admits right and
left Wiener-Hopf factorizations in $C_{N\times N}^\gamma$. In Section~\ref{sec:3}
we give the proof of Theorem~\ref{th:main} using an idea of B\"ottcher and
Silbermann \cite{BS80} of a decomposition of $\tr\log\{I-\sum_{k=0}^\infty G_{n,k}(b,c)\}$.
We show that this decomposition can be made for $a-\lambda$ uniform with respect
to $\lambda$ in a neighborhood of $\partial\Omega$. This actually implies that
(\ref{eq:BS}) is valid with $\gamma=\alpha=\beta$ and $a$ replaced by $a-\lambda$
uniformly with respect to $\lambda$ in a neighborhood of $\partial\Omega$.
Thus, Widom's arguments apply.
%%%%%%%%%%%%%%%%%%%%%%%%%%%%%%%%%%%%%%%%%%%%%%%%%%%%%%%%%%%%%%%%%%%%%%%%%%%%%%
\subsection{Higher order asymptotic trace formulas for Toeplitz matrices with
symbols from other smoothness classes}
Let us mention two other classes of symbols for which higher order asymptotic
formulas for $\tr f(T_n(a))$ are available.
%%%%%%%%%%%%%%%%%%%%%%%%%%%%%%%%%%%%%%%%%%%%%%%%%%%%%%%%%%%%%%%%%%%%%%%%%%%%%%
\begin{theorem}
Suppose $a$ is a continuous $N\times N$ matrix function
on the unit circle and $f$ is analytic on
$\mathrm{sp}\,T(a)\cup\mathrm{sp}\,T(\widetilde{a})$.
Let $\|\cdot\|$ be any matrix norm on $\C_{N\times N}$.
\begin{enumerate}
\item[{\rm (a)}] {\rm (see \cite{VMS03}).}
If $\gamma>1$ and
%%%
\[
\sum_{k=-\infty}^\infty \|a_k\|+
\sum_{k=-\infty}^\infty \|a_k\|^2|k|^\gamma<\infty,
\]
%%%
then {\rm (\ref{eq:Widom2})} is true with $o(1)$ replaced by $o(1/n^{\gamma-1})$.

\item[{\rm (b)}]
{\rm (see \cite[Corollary~1.6]{KZAA}).}
If $\alpha,\beta>0$, $\alpha+\beta>1$, and
\[
\sum_{k=1}^\infty \|a_{-k}\|k^\alpha+
\sum_{k=1}^\infty \|a_k\|k^\beta<\infty,
\]
then {\rm (\ref{eq:Widom2})} is true
with $o(1)$ replaced by $o(1/n^{\alpha+\beta-1})$.
\end{enumerate}
\end{theorem}
%%%%%%%%%%%%%%%%%%%%%%%%%%%%%%%%%%%%%%%%%%%%%%%%%%%%%%%%%%%%%%%%%%%%%%%%%%
\section{Wiener-Hopf factorization in decomposing algebras\\ of continuous functions}
\label{sec:2}
\subsection{Definitions and general theorems}
Let $\D$ be the open unit disk.
Let $\cR_-$ (resp. $\cR_+$) denote the set of all rational functions with poles
only in $\D$ (resp. in $(\C\cup\{\infty\})\setminus(\D\cup\T)$). Let $C_\pm$
be the closure of $\cR_\pm$ with respect to the norm of $C$. Suppose $\cA$
is a Banach algebra of continuous functions on $\T$ that contains $\cR_+\cup\cR_-$
and has the following property: if $a\in\cA$ and $a(t)\ne 0$ for all $t\in\T$,
then $a^{-1}\in\cA$. The sets $\cA_\pm:=\cA\cap C_\pm$ are subalgebras of $\cA$.
The algebra $\cA$ is said to be \textit{decomposing} if every function $a\in\cA$
can be represented in the form $a=a_-+a_+$ where $a_\pm\in\cA_\pm$.

Let $\cA$ be a decomposing algebra. A matrix function
$a\in\cA_{N\times N}$ is said to admit a \textit{right} (resp. \textit{left})
\textit{Wiener-Hopf} (WH) \textit{factorization in} $\cA_{N\times N}$ if it can
be represented in the form $a=a_-da_+$ (resp. $a=a_+da_-$), where
\[
a_\pm\in G(\cA_\pm)_{N\times N},
\quad
d(t)=\mathrm{diag}\{t^{\kappa_1},\dots,t^{\kappa_N}\},
\quad
\kappa_i\in\Z,
\quad
\kappa_1\le\dots\le\kappa_N.
\]
The integers $\kappa_i$ are usually called
the \textit{right} (resp. \textit{left}) \textit{partial indices} of $a$;
they can be shown to be uniquely determined by $a$. If $\kappa_1=\dots=\kappa_N=0$,
then the respective WH factorization is said to be \textit{canonical}.

The following result was obtained by Budjanu and Gohberg \cite[Theorem~4.3]{BG68}
and it is contained in \cite[Chap.~II, Corollary~5.1]{CG81} and in
\cite[Theorem~5.7']{LS87}.
%%%%%%%%%%%%%%%%%%%%%%%%%%%%%%%%%%%%%%%%%%%%%%%%%%%%%%%%%%%%%%%%%%%%%%%%%%
\begin{theorem}\label{th:factorization}
Suppose the following two conditions hold for the algebra $\cA$:
\begin{enumerate}
\item[{\rm (a)}]
the Cauchy singular integral operator
\[
(S\varphi)(t):=\frac{1}{\pi i}v.p.\int_\T\frac{\varphi(\tau)}{\tau-t}d\tau
\quad(t\in\T)
\]
is bounded on $\cA$;

\item[{\rm (b)}]
for any function $a\in\cA$, the operator $aS-SaI$ is compact on $\cA$.
\end{enumerate}
Then every matrix function $a\in\cA_{N\times N}$ such that $\det a(t)\ne 0$
for all $t\in\T$ admits a right and left WH factorizations in $\cA_{N\times N}$
(in general, with different sets of partial indices).
\end{theorem}
%%%%%%%%%%%%%%%%%%%%%%%%%%%%%%%%%%%%%%%%%%%%%%%%%%%%%%%%%%%%%%%%%%%%%%%%%%
Notice that (a) holds if and only if $\cA$ is a decomposing algebra.

The following theorem follows from a more general result due to Shubin
\cite{Shubin67}. Its proof can be found in \cite[Theorem~6.15]{LS87}.
%%%%%%%%%%%%%%%%%%%%%%%%%%%%%%%%%%%%%%%%%%%%%%%%%%%%%%%%%%%%%%%%%%%%%%%%%%
\begin{theorem}\label{th:stability}
Let $\cA$ be a decomposing algebra and let $\|\cdot\|$ be a norm in
the algebra $\cA_{N\times N}$. Suppose $a,c\in\cA_{N\times N}$ admit canonical
right and left WH factorizations in the algebra $\cA_{N\times N}$. Then for
every $\eps>0$ there exists a $\delta>0$ such that if $\|a-c\|<\delta$, then for
every canonical right WH factorization $a=a_-^{(r)}a_+^{(r)}$ and for every
canonical left WH factorization $a=a_+^{(l)}a_-^{(l)}$ one can choose
a canonical right WH factorization $c=c_-^{(r)}c_+^{(r)}$ and a canonical
left WH factorization $c=c_+^{(l)}c_-^{(l)}$ such that
%%%
\[
\begin{array}{lll}
\|a_\pm^{(r)}-c_\pm^{(r)}\|<\eps,
&\quad&
\|[a_\pm^{(r)}]^{-1}-[c_\pm^{(r)}]^{-1}\|<\eps,
\\[3mm]
\|a_\pm^{(l)}-c_\pm^{(l)}\|<\eps,
&\quad&
\|[a_\pm^{(l)}]^{-1}-[c_\pm^{(l)}]^{-1}\|<\eps.
\end{array}
\]
\end{theorem}
%%%%%%%%%%%%%%%%%%%%%%%%%%%%%%%%%%%%%%%%%%%%%%%%%%%%%%%%%%%%%%%%%%%%%%%%%%
\subsection{Wiener-Hopf factorization in H\"older-Zygmund spaces}
\begin{theorem}\label{th:Silbermann}
{\rm (see \cite[Section~6.25]{PS91}).}
Suppose $\gamma>0$. Then

\begin{enumerate}
\item[{\rm (a)}] $C^\gamma$ is a Banach algebra;
\item[{\rm (b)}] $a\in C^\gamma$ is invertible in $C^\gamma$ if and only if
$a(t)\ne 0$ for all $t\in\T$;
\item[{\rm (c)}] $S$ is bounded on $C^\gamma$;
\item[{\rm (d)}] for $a\in C^\gamma$, the operator $aI$ is bounded on $C^\gamma$
and the operator $aS-SaI$ is compact on $C^\gamma$.
\end{enumerate}
\end{theorem}

For $\gamma\notin\Z_+$, parts (c) and (d) are proved in
\cite[Section~7]{BG68} (see also \cite[Chap.~II, Section~6.2]{CG81}).
Note that a statement similar to (d) is proved in
\cite[Chap.~7, Theorem~4.3]{Peller03}.
%%%%%%%%%%%%%%%%%%%%%%%%%%%%%%%%%%%%%%%%%%%%%%%%%%%%%%%%%%%%%%%%%%%%%%%%%%
\begin{theorem}\label{th:Toep-fact}
Let $\gamma>0$ and $\Sigma$ be a compact set in the complex plane. Suppose
$a:\Sigma\to C_{N\times N}^\gamma$ is a continuous function and the Toeplitz
operators $T(a(\lambda))$ and $T([a(\lambda)]\widetilde{\hspace{3mm}})$ are
invertible on $H_N^2$ for all $\lambda\in\Sigma$. Then for every $\lambda\in\Sigma$
the function $a(\lambda):\T\to\C$ admits canonical right and left WH factorizations
\[
a(\lambda)=u_-(\lambda)u_+(\lambda)=v_+(\lambda)v_-(\lambda)
\]
in $C_{N\times N}^\gamma$. These factorizations can be chosen so that
$u_\pm,v_\pm,u_\pm^{-1},v_\pm^{-1}:\Sigma\to C_{N\times N}^\gamma$ are continuous.
\end{theorem}
%%%%%%%%%%%%%%%%%%%%%%%%%%%%%%%%%%%%%%%%%%%%%%%%%%%%%%%%%%%%%%%%%%%%%%%%%%
\begin{proof}
Fix $\lambda\in\Sigma$ and put $a:=a(\lambda)$. If $T(a)$ is invertible on
$H_N^2$, then $\det a(t)\ne 0$ for all $t\in\T$ (see, e.g.,
\cite[Chap.~VII, Proposition~2.1]{CG81}). Then, by \cite[Chap.~VII, Theorem~3.2]{CG81},
the matrix function $a$ admits a canonical right generalized factorization
in $L_N^2$, that is, $a=a_-a_+$, where $a_-^{\pm 1}\in (H_-^2)_{N\times N}$,
$a_+^{\pm 1}\in(H_+^2)_{N\times N}$ (and, moreover, the operator
$a_-Pa_-^{-1}I$ is bounded on $L_N^2$).

On the other hand, from Theorems~\ref{th:factorization} and \ref{th:Silbermann}
it follows that $a\in C_{N\times N}^\gamma$ admits a right WH factorization
$a=u_-du_+$ in $C_{N\times N}^\gamma$. Then
\[
u_\pm\in (C_\pm^\gamma)_{N\times N}\subset(H_\pm^2)_{N\times N},
\quad
u_\pm^{-1}\in (C_\pm^\gamma)_{N\times N}\subset(H_\pm^2)_{N\times N}.
\]
By the uniqueness of the partial indices in a right generalized factorization
in $L_N^2$ (see, e.g., \cite[Corollary~2.1]{LS87}), $d=1$.

Let us prove that $a$ admits also a canonical left WH factorization in the
algebra $C_{N\times N}^\gamma$. In view of Theorem~\ref{th:Silbermann}(b),
$a^{-1}\in C_{N\times N}^\gamma$. By \cite[Proposition~7.19(b)]{BS06},
the invertibility of $T(\widetilde{a})$ on $H_N^2$ is equivalent to the
invertibility of $T(a^{-1})$ on $H_N^2$. By what has just been proved, there
exist $f_\pm\in G(C_\pm^\gamma)_{N\times N}$ such that $a^{-1}=f_-f_+$.
Put $v_\pm:=f_\pm^{-1}$. Then $v_\pm\in G(C_\pm^\gamma)_{N\times N}$ and
$a=v_+v_-$ is a canonical left WH factorization in $C_{N\times N}^\gamma$.

We have proved that for each $\lambda\in\Sigma$ the matrix function
$a(\lambda):\T\to\C$ admits canonical right and left WH factorizations in
$C_{N\times N}^\gamma$. By Theorem~\ref{th:stability}, these factorizations
can be chosen so that the factors $u_\pm,v_\pm$ and their inverses
$u_\pm^{-1},v_\pm^{-1}$ are continuous functions from $\Sigma$ to
$C_{N\times N}^\gamma$.
\end{proof}
%%%%%%%%%%%%%%%%%%%%%%%%%%%%%%%%%%%%%%%%%%%%%%%%%%%%%%%%%%%%%%%%%%%%%%%%%%
\section{Proof of the main result}\label{sec:3}
%%%%%%%%%%%%%%%%%%%%%%%%%%%%%%%%%%%%%%%%%%%%%%%%%%%%%%%%%%%%%%%%%%%%%%%%%%
\subsection{The B\"ottcher-Silbermann decomposition}\label{sec:BSdecomposition}
The following result from \cite[Section~6.16]{BS83}, \cite[Section~10.34]{BS06}
is the basis for our asymptotic analysis.
%%%%%%%%%%%%%%%%%%%%%%%%%%%%%%%%%%%%%%%%%%%%%%%%%%%%%%%%%%%%%%%%%%%%%%%%%%%
\begin{lemma}\label{le:BS}
Suppose $a\in L_{N\times N}^\infty$ satisfies the following hypotheses:

\begin{enumerate}
\item[{\rm (i)}]
there are two factorizations $a=u_-u_+=v_+v_-$, where
$u_+,v_+\in G(H_+^\infty)_{N\times N}$ and
$u_-,v_-\in G(H_-^\infty)_{N\times N}$;

\item[{\rm (ii)}]
$u_-\in C_{N\times N}$ or $u_+\in C_{N\times N}$.
\end{enumerate}

\noindent
Define the functions $b$, $c$ by $b:=v_-u_+^{-1}$, $c:=u_-^{-1}v_+$
and the matrices $G_{n,k}(b,c)$ by {\rm (\ref{eq:defGnk})}.
Suppose for all sufficiently large $n$ {\rm (}say, $n\ge N_0${\rm)}
there exists a decomposition
%%%
\begin{equation}\label{eq:decomposition}
\tr\log\left\{I-\sum_{k=0}^\infty G_{n,k}(b,c)\right\}=-\tr H_n+s_n
\end{equation}
%%%%
where $\{H_n\}_{n=N_0}^\infty$ is a sequence of $N\times N$ matrices
and $\{s_n\}_{n=N_0}^\infty$ is a sequence of complex numbers. If
$\sum_{n=N_0}^\infty |s_n|<\infty$, then there exist a constant
$\widetilde{E}(a)\ne 0$ depending on $\{H_n\}_{n=N_0}^\infty$ and
arbitrarily chosen $N\times N$ matrices $H_1,\dots,H_{N_0-1}$ such
that for all $n\ge N_0$,
\[
\log \det T_n(a)=(n+1)\log G(a)+\tr(H_1+\dots +H_n)+
\log \widetilde{E}(a)+\sum_{k=n+1}^\infty s_k,
\]
where the constant $G(a)$ is given by {\rm (\ref{eq:defG})}.
\end{lemma}
%%%%%%%%%%%%%%%%%%%%%%%%%%%%%%%%%%%%%%%%%%%%%%%%%%%%%%%%%%%%%%%%%%%%%%%%%%
\subsection{The best uniform approximation}
Let $\cP^n$ be the set of all Laurent polynomials of the form
\[
p(t)=\sum_{j=-n}^n\alpha_j t^j,\quad\alpha_j\in\C,\quad t\in\T.
\]
By the Chebyshev theorem (see, e.g., \cite[Section~2.2.1]{Timan63}),
for $f\in C$ and $n\in\N$, there is a Laurent polynomial $p_n(f)\in\cP^n$
such that
%%%
\begin{equation}\label{eq:approx1}
\|f-p_n(f)\|_\infty=\inf_{p\in\cP^n}\|f-p\|_\infty.
\end{equation}
%%%
This polynomial $p_n(f)$ is called a polynomial of best uniform
approximation.

By the Jackson-Ahiezer-Stechkin theorem (see, e.g., \cite[Section~5.1.4]{Timan63}),
if $f$ has a bounded derivative $f^{(m)}$ of order $m$ on $\T$, then
for $n\in\N$,
%%%
\begin{equation}\label{eq:approx2}
\inf_{p\in\cP^n}\|f-p\|_\infty\le
\frac{C_m}{(n+1)^m}\omega_2\left(f^{(m)},\frac{1}{n+1}\right),
\end{equation}
%%%
where the constant $C_m$ depends only on $m$.

From (\ref{eq:approx1}) and (\ref{eq:approx2}) it follows that
if $f\in C^\gamma$ and $n\in\N$, where $\gamma=m+\delta$ with $m\in\Z_+$
and $\delta\in(0,1]$, then there is a $p_n(f)\in\cP^n$ such that
%%%
\begin{equation}\label{eq:approx3}
\|f-p_n(f)\|_\infty\le\frac{C_m}{(n+1)^m}\omega_2\left(f^{(m)},\frac{1}{n+1}\right)
\le
\frac{C_m[f^{(m)}]_\delta}{(n+1)^{m+\delta}}
\le C_m\frac{\|f\|_\gamma}{n^\gamma}.
\end{equation}
%%%%%%%%%%%%%%%%%%%%%%%%%%%%%%%%%%%%%%%%%%%%%%%%%%%%%%%%%%%%%%%%%%%%%%%%%%
\subsection{Norms of truncations of Toeplitz and Hankel operators}
Let $X$ be a Banach space.
For definiteness, let the norm of $a=[a_{ij}]_{i,j=1}^N$ in $X_{N\times N}$
be given by $\|a\|_{X_{N\times N}}=\max\limits_{1\le i,j\le N}\|a_{ij}\|_X$.
We will simply write $\|a\|_\infty$ and $\|a\|_\gamma$ instead of
$\|a\|_{L_{N\times N}^\infty}$ and $\|a\|_{C_{N\times N}^\gamma}$, respectively.
Denote by $\|A\|$ the norm of a bounded linear operator $A$ on $H_N^2$.

A slightly less precise version of the following statement was used
in the proof of \cite[Theorem~10.35(ii)]{BS06}.
%%%%%%%%%%%%%%%%%%%%%%%%%%%%%%%%%%%%%%%%%%%%%%%%%%%%%%%%%%%%%%%%%%%%%%%%%%%
\begin{proposition}\label{pr:trunc}
Let $\alpha,\beta>0$. Suppose $b=v_-u_+^{-1}$ and $c=u_-^{-1}v_+$, where
\[
u_+\in G(C^\alpha\cap H_+^\infty)_{N\times N},
\quad
u_-\in G(C^\beta\cap H_-^\infty)_{N\times N},
\quad
v_\pm\in G(H_\pm^\infty)_{N\times N}.
\]
Then there exist positive constants $M_\alpha$ and $M_\beta$
depending only on $N$ and $\alpha$ and $\beta$, respectively,
such that for all $n\in\N$,
\[
\|Q_nT(b)P_0\|\le \frac{M_\alpha}{n^\alpha}\|v_-\|_\infty\|u_+^{-1}\|_\alpha,
\quad
\|Q_nH(b)\|\le\frac{M_\alpha}{n^\alpha}\|v_-\|_\infty\|u_+^{-1}\|_\alpha,
\]
\[
\|P_0T(c)Q_n\|\le\frac{M_\beta}{n^\beta}\|v_+\|_\infty\|u_-^{-1}\|_\beta,
\quad
\|H(\widetilde{c})Q_n\|\le\frac{M_\beta}{n^\beta}\|v_+\|_\infty\|u_-^{-1}\|_\beta.
\]
\end{proposition}
%%%%%%%%%%%%%%%%%%%%%%%%%%%%%%%%%%%%%%%%%%%%%%%%%%%%%%%%%%%%%%%%%%%%%%%%%%
\begin{proof}
Since $b=v_-u_+^{-1}$, $c=u_-^{-1}v_+$ and $v_\pm,u_\pm\in G(H_\pm^\infty)_{N\times N}$,
one has
%%%
\begin{eqnarray}
&&
Q_nT(b)P_0 = Q_nT(v_-)Q_nT(u_+^{-1})P_0,
\label{eq:trunc1}
\\
&&
Q_nH(b) = Q_nT(v_-)Q_nH(u_+^{-1}),
\label{eq:trunc2}
\\
&&
P_0T(c)Q_n=P_0T(u_-^{-1})Q_nT(v_+)Q_n,
\label{eq:trunc3}
\\
&&
H(\widetilde{c})Q_n=H(\widetilde{u_-^{-1}})Q_nT(v_+)Q_n.
\label{eq:trunc4}
\end{eqnarray}
%%%
Let $p_n(u_+^{-1})$ and $p_n(u_-^{-1})$ be the polynomials in $\cP_{N\times N}^n$
of best uniform approximation of $u_+^{-1}$ and $u_-^{-1}$, respectively.
Obviously,
\[
\begin{array}{lll}
Q_nT[p_n(u_+^{-1})]P_0=0,
&\quad
&Q_nH[p_n(u_+^{-1})]=0,\\[3mm]
P_0T[p_n(u_-^{-1})]Q_n=0,
&\quad
&H[(p_n(u_-^{-1}))\widetilde{\hspace{2mm}}]Q_n=0.
\end{array}
\]
Then from (\ref{eq:approx3}) it follows that
%%%
\begin{eqnarray}
\label{eq:trunc5}
\|Q_nT(u_+^{-1})P_0\|
&=&
\|Q_nT[u_+^{-1}-p_n(u_+^{-1})]P_0\|
\\
\nonumber
&\le&
\|P\|\,\|u_+^{-1}-p_n(u_+^{-1})\|_\infty
\le
\frac{M_\alpha}{n_\alpha}\|u_+^{-1}\|_\alpha
\end{eqnarray}
%%%
and similarly
%%%
\begin{eqnarray}
&&
\|Q_n H(u_+^{-1})\|
\le
\frac{M_\alpha}{n^\alpha}\|u_+^{-1}\|_\alpha,
\label{eq:trunc6}
\\
&&
\|P_0T(u_-^{-1})Q_n\|\le\frac{M_\beta}{n^\beta}\|u_-^{-1}\|_\beta,
\label{eq:trunc7}
\\
&&
\|H(\widetilde{u_-^{-1}})Q_n\|\le\frac{M_\beta}{n^\beta}\|u_-^{-1}\|_\beta,
\label{eq:trunc8}
\end{eqnarray}
%%%
where $M_\alpha$ and $M_\beta$ depend only on $\alpha,\beta$ and $N$.
Combining (\ref{eq:trunc1}) and (\ref{eq:trunc5}), we get
\[
\|Q_nT(b)P_0\|
\le
\|T(v_-)\|\,\|Q_nT(u_-^{-1})P_0\|
\le
\frac{M_\alpha}{n^\alpha}\|v_-\|_\infty\|u_+^{-1}\|_\alpha.
\]
All other assertions follow from (\ref{eq:trunc2})--(\ref{eq:trunc4})
and (\ref{eq:trunc6})--(\ref{eq:trunc8}).
\end{proof}
%%%%%%%%%%%%%%%%%%%%%%%%%%%%%%%%%%%%%%%%%%%%%%%%%%%%%%%%%%%%%%%%%%%%%%%%%%
\subsection{The key estimate}
The following proposition shows that a decomposition of Lemma~\ref{le:BS}
exists.
%%%%%%%%%%%%%%%%%%%%%%%%%%%%%%%%%%%%%%%%%%%%%%%%%%%%%%%%%%%%%%%%%%%%%%%%%%
\begin{proposition}\label{pr:key}
Suppose the conditions of Proposition~{\rm\ref{pr:trunc}} are fulfilled.
If $p\in\N$, then there exists a constant $C_p\in(0,\infty)$ depending only
on $p$ such that
%%%
\begin{eqnarray*}
&&
\left|\,
\tr\log\left\{I-\sum_{k=0}^\infty G_{n,k}(b,c)\right\}
+\tr\left[\sum_{j=1}^{p-1}\frac{1}{j}\left(\sum_{k=0}^{p-j-1}G_{n,k}(b,c)\right)^j\right]
\right|
\\
&&
\le
C_p\left(\frac{M_\alpha M_\beta}{n^{\alpha+\beta}}
\|u_+^{-1}\|_\alpha\|u_-^{-1}\|_\beta\|v_-\|_\infty\|v_+\|_\infty\right)^p
\end{eqnarray*}
%%%
for all
$n>\big(M_\alpha M_\beta \|u_+^{-1}\|_\alpha\|u_-^{-1}\|_\beta
\|v_-\|_\infty\|v_+\|_\infty\big)^{1/(\alpha+\beta)}$.
\end{proposition}
%%%%%%%%%%%%%%%%%%%%%%%%%%%%%%%%%%%%%%%%%%%%%%%%%%%%%%%%%%%%%%%%%%%%%%%%%%
\begin{proof}
From Proposition~\ref{pr:trunc} it follows that
\[
\|G_{n,k}(b,c)\|\le\left[\frac{M_\alpha M_\beta}{n^{\alpha+\beta}}
\|u_+^{-1}\|_\alpha\|u_-^{-1}\|_\beta\|v_-\|_\infty\|v_+\|_\infty\right]^{k+1}
\]
for all $k\in\Z_+$ and $n\in\N$. If
$n>\big(M_\alpha M_\beta \|u_+^{-1}\|_\alpha\|u_-^{-1}\|_\beta
\|v_-\|_\infty\|v_+\|_\infty\big)^{1/(\alpha+\beta)}$,
then the expression in the brackets is less than $1$. In view of these
observations the proof can be developed as in \cite[Proposition~3.3]{KZAA}.
\end{proof}
%%%%%%%%%%%%%%%%%%%%%%%%%%%%%%%%%%%%%%%%%%%%%%%%%%%%%%%%%%%%%%%%%%%%%%%%%%
Theorem~\ref{th:BS}(b) follows from the above statement and Lemma~\ref{le:BS}.
In the next section we will use the partial case $p=1$ of Proposition~\ref{pr:key}
as the key ingredient of the proof of our main result.
%%%%%%%%%%%%%%%%%%%%%%%%%%%%%%%%%%%%%%%%%%%%%%%%%%%%%%%%%%%%%%%%%%%%%%%%%%
\subsection{Proof of Theorem~\ref{th:main}}
Suppose $\gamma>1/2$ and
$\lambda\notin\mathrm{sp}\,T(a)\cup\mathrm{sp}\,T(\widetilde{a})$. Then
\[
T(a)-\lambda I=T(a-\lambda),
\quad
T(\widetilde{a})-\lambda I=T([a-\lambda]\widetilde{\hspace{3mm}})
\]
are invertible on $H_N^2$. Since $a-\lambda$ is continuous with respect
to $\lambda$ as a function from a closed neighborhood $\Sigma$ of
$\partial\Omega$ to $C_{N\times N}^\gamma$, in view of Theorem~\ref{th:Toep-fact},
for each $\lambda\in\Sigma$, the function $a-\lambda:\T\to\C$ admits
canonical right and left WH factorizations
$a-\lambda=u_-(\lambda)u_+(\lambda)=v_+(\lambda)v_-(\lambda)$ in
$C_{N\times N}^\gamma$ and these factorizations can be chosen so that the
factors $u_\pm$, $v_\pm$ and their inverses $u_\pm^{-1}$, $v_\pm^{-1}$ are
continuous from $\Sigma$ to $C_{N\times N}^\gamma$. Then
\[
A_\Sigma:=\max_{\lambda\in\Sigma}\big(
\|u_+^{-1}(\lambda)\|_\gamma
\|u_-^{-1}(\lambda)\|_\gamma
\|v_-(\lambda)\|_\gamma
\|v_+(\lambda)\|_\gamma
\big)<\infty.
\]
Put $b=v_-u_+^{-1}$ and $c=u_-^{-1}v_+$. From Proposition~\ref{pr:key}
with $p=1$ it follows that there exists $C_1\in(0,\infty)$ such that
%%%
\begin{eqnarray}
\label{eq:proof1}
&&
\left|\,\tr\log\left\{I-\sum_{k=0}^\infty G_{n,k}(b(\lambda),c(\lambda))\right\}\right|
\\
&&
\le
\frac{C_1 M_\gamma^2}{n^{2\gamma}}
\|u_+^{-1}(\lambda)\|_\gamma
\|u_-^{-1}(\lambda)\|_\gamma
\|v_-(\lambda)\|_\infty
\|v_+(\lambda)\|_\infty
\nonumber\\
&&
\le
\frac{C_1 M_\gamma^2 A_\Sigma}{n^{2\gamma}}
\nonumber
\end{eqnarray}
%%%
for all $n>(M_\gamma^2 A_\Sigma)^{1/(2\gamma)}$ and all $\lambda\in\Sigma$.
Obviously
%%%
\begin{equation}\label{eq:proof2}
\sum_{k=n+1}^\infty\frac{1}{k^{2\gamma}}=O(1/n^{2\gamma-1}).
\end{equation}
%%%
From Lemma~\ref{le:BS} and (\ref{eq:proof1})--(\ref{eq:proof2})
it follows that there is a function $\widetilde{E}(a,\cdot):\Sigma\to\C\setminus\{0\}$
such that
%%%
\begin{equation}\label{eq:proof3}
\log\det T_n(a-\lambda)=(n+1)\log G(a-\lambda)+
\log \widetilde{E}(a,\lambda)+O(1/n^{2\gamma-1})
\end{equation}
%%%
as $n\to\infty$ and this holds uniformly with respect to $\lambda\in\Sigma$.
Theorem~\ref{th:BS}(c) implies that $T(a-\lambda)T([a-\lambda]^{-1})-I$
is of trace class and
%%%
\begin{equation}\label{eq:proof4}
\widetilde{E}(a,\lambda)=\det T(a-\lambda)T([a-\lambda]^{-1})
\end{equation}
%%%
for all $\lambda\in\Sigma$. Combining (\ref{eq:proof3}) and (\ref{eq:proof4}),
we deduce that
\[
\log\det T_n(a-\lambda)=(n+1)\log G(a-\lambda)+
\log \det T(a-\lambda)T([a-\lambda]^{-1})
+O(1/n^{2\gamma-1})
\]
as $n\to\infty$ uniformly with respect to $\lambda\in\Sigma$.
Hence, one can differentiate both sides of the last formula with respect
to $\lambda$, multiply by $f(\lambda)$, and integrate over $\partial\Omega$.
The proof is finished by a literal repetition of Widom's proof of
Theorem~\ref{th:Widom2} (see \cite[p.~21]{Widom76} or \cite[Section~10.90]{BS06})
with $o(1)$ replaced by $O(1/n^{2\gamma-1})$.
\qed
%%%%%%%%%%%%%%%%%%%%%%%%%%%%%%%%%%%%%%%%%%%%%%%%%%%%%%%%%%%%%%%%%%%%%%%%%%


\begin{thebibliography}{99}
\bibitem{BG05}
A.~B\"ottcher and S.~M.~Grudsky,
\textit{Spectral Properties of Banded Toeplitz Operators}.
SIAM, Philadelphia, PA, 2005.

\bibitem{BS80}
A.~B\"ottcher and B.~Silbermann,
\textit{Notes on the asymptotic behavior of block Toeplitz matrices and determinants}.
Math. Nachr., \textbf{98} (1980), 183--210.

\bibitem{BS83}
A.~B\"ottcher and B.~Silbermann,
\textit{Invertibility and Asymptotics of Toeplitz Matrices}.
Akademie-Verlag, Berlin, 1983.

\bibitem{BS99}
A.~B\"ottcher and B.~Silbermann,
\textit{Introduction to Large Truncated Toeplitz Matrices}.
Springer-Verlag, New York, 1999.

\bibitem{BS06}
A.~B\"ottcher and B.~Silbermann,
\textit{Analysis of Toeplitz Operators}. 2nd edition.
Springer-Verlag, Berlin, 2006.

\bibitem{BG68}
M.~S.~Budjanu and I.~C.~Gohberg,
\textit{General theorems on the factorization of matrix-valued functions. II.
Some tests and their consequences}.
Amer. Math. Soc. Transl. (2), \textbf{102} (1973), 15--26.

\bibitem{CG81}
K.~F.~Clancey and I.~Gohberg,
\textit{Factorization of Matrix Functions and Singular Integral Operators}.
Birkh\"auser Verlag, Basel, 1981.

\bibitem{Ehrhardt03}
T.~Ehrhardt,
\textit{A new algebraic approach to the Szeg\H{o}-Widom limit theorem}.
Acta Math. Hungar., \textbf{99} (2003), 233--261.

\bibitem{GK69}
I.~C.~Gohberg and M.~G.~Krein,
\textit{Introduction to the Theory of Linear Nonselfadjoint Operators}.
AMS, Providence, RI, 1969.

\bibitem{GS58}
U.~Grenander and G.~Szeg\H{o},
\textit{Toeplitz Forms and Their Applications}.
University of California Press, Berkeley, Los Angeles, 1958.

\bibitem{KZAA}
A.~Yu.~Karlovich,
\textit{Asymptotics of determinants and traces of Toeplitz matrices with symbols in
weighted Wiener algebras}.
Z. Anal. Anwendungen, \textbf{26} (2007), 43--56.

\bibitem{Krein66}
M.~G.~Krein,
\textit{Certain new Banach algebras and theorems of the type
of the Wiener-L\'evy theorems for series and Fourier integrals}.
%(in Russian).
%Mat. Issled. 1 (1966), no. 1, 82--109.
Amer. Math. Soc. Transl. (2), \textbf{93} (1970), 177--199.

\bibitem{LS87}
G.~S.~Litvinchuk and I.~M.~Spitkovsky,
\textit{Factorization of Measurable Matrix Functions}.
Birkh\"auser Verlag, Basel, 1987.

\bibitem{Peller03}
V.~V.~Peller,
\textit{Hankel Operators and Their Applications}.
Springer, New York, 2003.

\bibitem{PS91}
S.~Pr\"ossdorf and B.~Silbermann,
\textit{Numerical Analysis for Integral and Related Operator Equations}.
Birkh\"auser Verlag, Basel, 1991.

\bibitem{ST87}
H.-J.~Schmeisser and H.~Triebel,
\textit{Topics in Fourier Analysis and Function Spaces}.
John Wiley \& Sons, Chichester, 1987.

\bibitem{Shubin67}
M.~A.~Shubin,
\textit{Factorization of matrix functions depending on a parameter in normed rings
and related problems of the theory of Noetherian operators}.
Matem. Sbornik, \textbf{73(115)} (1967), 610--629 (in Russian).

\bibitem{Simon05}
B.~Simon,
\textit{Orthogonal Polynomials on the Unit Circle. Part 1}.
AMS, Providence, RI, 2005.

\bibitem{Timan63}
A.~F.~Timan,
\textit{Theory of Approximation of Functions of a Real Variable}.
Pergamon Press, Oxford, 1963.

\bibitem{VMS03}
V.~A.~Vasil'ev, E.~A.~Maksimenko, and I.~B.~Simonenko,
\textit{On a Szeg\H{o}-Widom limit theorem}.
Dokl. Akad. Nauk, \textbf{393} (2003),  307--308 (in Russian).

\bibitem{Widom76}
H.~Widom,
\textit{Asymptotic behavior of block Toeplitz matrices and determinants. II}.
Advances in Math., \textbf{21} (1976), 1--29.
\end{thebibliography}
\end{document}